\documentclass[12pt]{article}
\usepackage{a4wide}
\usepackage[french,english]{babel}
\usepackage[utf8]{inputenc}
\usepackage[T1]{fontenc}
\usepackage{amsthm}
\usepackage{amssymb}
\usepackage{amsmath}
\usepackage{stmaryrd}
\usepackage{makeidx}
\usepackage{graphicx}
\usepackage{caption}
\theoremstyle{plain}
\usepackage[colorlinks=true,breaklinks=true,linkcolor=black]{hyperref} 
\usepackage{amsfonts}
\newtheorem{theorem}{Theorem}[section] 

\newtheorem{corollary}[theorem]{Corollary}
\newtheorem{definition}[theorem]{Definition}
\newtheorem{lemma}[theorem]{Lemma}
\date{}

\makeindex

\usepackage{mathrsfs}
\usepackage[all]{xy}

\theoremstyle{definition}

\newtheorem{defn/prop}[theorem]{Définition/Proposition}
\newtheorem{defn/thm}[theorem]{Définition/Théorème}
\newtheorem{example}[theorem]{Example}

\newtheorem{rque}[theorem]{Remark}
\usepackage{geometry}
\title{Gindikin-Karpelevich finiteness for Kac-Moody groups over local fields}

\author{Auguste \textsc{Hébert} \\Univ Lyon, UJM-Saint-Etienne CNRS\\ UMR 5208 CNRS, F-42023, SAINT-ETIENNE, France}
\makeatletter \@addtoreset{figure}{section}\makeatother

\newcommand{\R}{\mathbb{R}}
\newcommand{\A}{\mathbb{A}}
\newcommand{\N}{\mathbb{N}}
\newcommand{\Z}{\mathbb{Z}}

\newcommand{\C}{\mathbb{C}}
\newcommand{\Ne}{\mathbb{N}^*}

\newcommand{\I}{\mathcal{I}}
\newcommand{\T}{\mathcal{T}}
\newcommand{\Id}{\mathrm{Id}}

\begin{document}

\maketitle

\section{Introduction}
The classical Gindikin-Karpelevich formula was introduced in 1962 by Simon Gindikin and Fridrikh Karpelevich. It applies to real semi-simple Lie groups and it enables to compute certain Plancherel densities for semi-simple Lie groups. This formula was established in the non-archimedean case by Robert Langlands in \cite{langlands1971euler}. In 2014, in \cite{braverman2014affine}, Alexander Braverman, Howard Garland, David Kazhdan and Manish Patnaik obtained a generalisation of this formula in the affine Kac-Moody case.

Let $\mathcal{F}$ be a ultrametric field. Let $\mathcal{O}$ be its ring of integers, $\pi$ be a generator of the maximal ideal of $\mathcal{O}$ and $q$ denote the cardinal of the residue field $\mathcal{O}/\pi \mathcal{O}$. Let $\mathbf{G}$ be a split Kac-Moody group over $\mathcal{F}$. Let $\textbf{T}$ be a maximal split torus of $\textbf{G}$. Choose
  a pair $\textbf{B}$, $\textbf{B}^-$ of opposite Borel subgroups such that
   $\textbf{B}\cap \textbf{B}^-=\textbf{T}$ and $\textbf{U}$, $\textbf{U}^-$ be their
    unipotent radicals. We use boldface letters to denote schemes : $\textbf{G},\textbf{T},\ldots$ and their sets of $\mathcal{F}$-points are denoted by $G$,$T$,...  Let $\Lambda$ and $\Lambda^\vee$ be the root lattice and the coroot lattice of $\textbf{T}$. Let $R$ be the set of roots of $\textbf{G}$ and $R^+$ be its set of positive roots. Let $K=\textbf{G}(\mathcal{O})$.  Let $S^\vee$ be the set of simple coroots ($S^\vee$ depends on the choice of $\textbf{B}$). For a coroot $\nu^\vee=\sum_{s^\vee\in S^\vee} n_{s^\vee} s^\vee\in \Lambda^\vee$, one sets $|\nu^\vee|=\sum n_{s^\vee}$. Let $\Lambda^\vee_+=\bigoplus_{s^\vee\in S^\vee}\N s^\vee$ and $\Lambda^\vee_-=-\Lambda^\vee_+$. Let $\C[[\Lambda^\vee]]$ be the Looijenga's coweight algebra of $\textbf{G}$, with generators $e^{\lambda^\vee}$, for $\lambda^\vee\in \Lambda^\vee$. Alexander Braverman, Howard Garland, David Kazhdan and Manish Patnaik proved that when $\textbf{G}$ is the affine Kac-Moody group associated to a simply-connected semi-simple split group, we have the following formula:

    \begin{equation}\label{eqGK formula}
    \sum_{\mu^\vee\in\Lambda^\vee}|K\backslash K\pi^{\lambda^\vee}U^-\cap
K\pi^{\lambda^\vee-\mu^\vee}U|e^{\lambda^\vee-\mu^\vee}q^{|\lambda^\vee-\mu^\vee|}=\frac{1}{H_0}\prod_{\alpha\in R_+}(\frac{1-q^{-1}e^{-\alpha^\vee}}{1-e^{-\alpha^\vee}})^{m_\alpha}\end{equation}

where for $\alpha\in R_+$, $m_\alpha$ denotes the multiplicity of the coroot $\alpha^\vee$ in the Lie algebra $\mathfrak{g}$ of $\textbf{G}$, and $H_0$ is some  term (there is a formula describing $H_0$ in \cite{braverman2014affine}) depending on $\mathfrak{g}$. When $\textbf{G}$ is a reductive group, $H_0=1$, the $m_\alpha$ are equal to $1$ and this formula is equivalent to Gindikin-Karpelevich formula. 

\bigskip

The aim of this article is to show the following four theorems:

\medskip

\textbf{Theorem~\ref{thm inclusion}}:
Let $\mu^\vee\in \Lambda^\vee$. Then if $\mu^\vee\notin \Lambda^\vee_-$, $K\pi^{\lambda^\vee}U^- \cap K\pi^{\lambda^\vee+\mu^\vee} U$ is empty for all $\lambda^\vee\in \Lambda^\vee$.
 If $\mu^\vee\in \Lambda^\vee_+$, then for $\lambda^\vee\in \Lambda^\vee$ sufficiently
  dominant, $K\pi^{\lambda^\vee}U^- \cap K\pi^{\lambda^\vee+\mu^\vee} U\subset K\pi^{\lambda^\vee} K\cap K\pi^{\lambda^\vee+\mu^\vee}U$. 

\medskip
\textbf{Corollary~\ref{cor finitude des boules}}:
 Let $\mu^\vee\in \Lambda^\vee$. For all $\lambda\in \Lambda^\vee$ , $ K\backslash K\pi^{\lambda^\vee} K\cap K\pi^{\lambda^\vee+\mu^\vee}U$ is finite and is empty if $\mu^\vee\notin \Lambda^\vee_-$.

\medskip
\textbf{Theorem~\ref{thm invariance des cardinaux}}: 
Let $\mu^\vee\in \Lambda^\vee_-$ and $\lambda^\vee\in \Lambda^\vee$. Then $|K\backslash K\pi^{\lambda^\vee}U^- \cap K\pi^{\lambda^\vee+\mu^\vee} U|=|K\backslash K\pi^{0}U^- \cap K\pi^{\mu^\vee} U|$ and these sets are finite.

\medskip
\textbf{Theorem~\ref{thm égalité des ensembles bis}}:
 Let $\mu^\vee\in \Lambda^\vee$. Then for $\lambda^\vee\in \Lambda^\vee$ sufficiently dominant, $ K\pi^{\lambda^\vee}U^- \cap K\pi^{\lambda^\vee+\mu^\vee} U=K\pi^{\lambda^\vee} K\cap K\pi^{\lambda^\vee+\mu^\vee}U$.

\medskip

These theorems correspond to Theorem 1.9 of \cite{braverman2014affine} and this is a positive answer to Conjecture 4.4 of \cite{braverman2012rep}. 

  Theorem~\ref{thm invariance des cardinaux} is named "Gindikin-Karpelevich finiteness" in \cite{braverman2016iwahori} and it enables to make sense to the left side of formula~\ref{eqGK formula}. Thus this could be a first step to a generalization of Gindikin-Karpelevich formula to general split Kac-Moody groups. However, we do not know yet how to express the term $H_0$ when $\textbf{G}$ is not affine.

\medskip

The main tool that we use to prove these theorems is the masure $\I$ associated to $\textbf{G}$, which was defined by Stéphane Gaussent and  Guy Rousseau in \cite{gaussent2014spherical} and generalized by Guy Rousseau in \cite{rousseau2017almost}. This masure is a geometric object similar to a Bruhat-Tits building. It is a set covered by apartments isomorphic to a standard apartment $\A$ on which $\textbf{G}$ acts. The standard apartment $\A$ is an affine space with a structure obtained from a lattice $Y$ containing $\Lambda^\vee$ (which can be thought of as $\Lambda^\vee$ in a first approximation) and from a set of hyperplanes called walls, defined by the roots of $\textbf{G}$. In general, $\I$ is not a building: there may be two points such that no apartment of $\I$ contains these two points simultaneously. However some properties of euclidean buildings remain and, for example one can still define the retraction onto an apartment centred at a sector-germ. We can always define a fundamental vectorial chamber $C^v_f$ in $\A$ and its opposite, which gives rise to two retractions: $\rho_{+\infty}$ and $\rho_{-\infty}$. We also have a "vectorial distance" $d^v$ defined on a subset of $\I^2$. These applications enable to give other expressions to the sets defined above: for all $\lambda^\vee\in \Lambda^\vee$, $K\backslash K \pi^{\lambda^\vee} U^{-}$ corresponds to $\rho_{-\infty}^{-1}(\{\lambda^\vee\})$, $K\backslash K\pi^{\lambda^\vee} U$ corresponds to $\rho_{+\infty}^{-1}(\{\lambda^\vee\})$ and $K\backslash K\pi^{\lambda^\vee} K$ corresponds to $S^v(0,\lambda^\vee)=\{x\in \I|\ d^v(0,x)\mathrm{\ is\ defined\ and\ }d^v(0,x)=\lambda^\vee\}$. This enables to use properties of "Hecke paths", which are more or less the images of segments in the masure by retractions, to prove these theorems. These paths were first defined by Misha Kapovich and John J.Millson in \cite{MR2415306}. 

Actually, when we will write Theorem~\ref{thm inclusion} and Theorem~\ref{thm égalité des ensembles bis} using the masure, we will show that these inclusion or equality are true modulo $K$. But as if $X\subset G$ is invariant by left  multiplication by $K$, $X=\bigcup_{x\in K\backslash X} Kx$, this will be sufficient.

The framework of masures is a bit more general than that of split Kac-Moody groups over local non-archimedean field. For example, we might also apply this to almost split Kac-Moody groups over local non-archimedean field (see \cite{rousseau2017almost}). 

In the sequel, the results are not stated as in this introduction. Their statements use retractions and vectorial distance. They are also a bit more general: they take into account the inessential part of the standard apartment $\A$, which will be defined later.  Corollary~\ref{cor finitude des boules}, as it is stated in this introduction was proved in Section 5 of \cite{gaussent2014spherical} and it is slightly generalized in the following. The lattice $\Lambda^\vee$ will be denoted by $Q^\vee$.

\medskip

In section~\ref{sect general frameworks} we set the general frameworks, we set the notation, we define masures and we prove that $Y^{++}=Y\cap \overline{C^v_f}$ is a finitely generated monoid, which will be useful to prove Theorem~\ref{thm égalité des ensembles bis}.

 In section~\ref{sect preliminaries} we first define two applications: $T_\nu:\I\rightarrow\R_+$ and $y_\nu:\I\rightarrow \A$, for a fixed $\nu\in C_f^v$. For $x\in\I$, $T_\nu(x)$ measures the distance between the point $x$ and the apartment $\A$ along $\R_+\nu$ and $y_\nu(x)$ defines the projection of $x$ on $\A$ along $\R_+\nu$. We also determine the antecedents of some kinds of paths by $\rho_{-\infty}$. 
 
 In section~\ref{sect bounding of T}, we show that $T_\nu$ is bounded by some function of $\rho_{+\infty}-\rho_{-\infty}$ and we deduce Theorem~\ref{thm inclusion}. 
 
 In section~\ref{sect translations}, we study some kinds of translations of $\I$, which enables us to show Corollary~\ref{cor finitude des boules} and Theorem~\ref{thm invariance des cardinaux}.

 In section~\ref{sect proof of final theorem}, we use the tools of the preceding sections to show Theorem~\ref{thm égalité des ensembles bis}.

\paragraph{Acknowledgements}
I thank Stéphane Gaussent for suggesting me this problem and for helping me improving the manuscript. I also thank Guy Rousseau for his remarks on a previous version of this paper.

\paragraph{Funding}
This work was supported by the ANR grants ANR-15-CE40-0012.

\section{General frameworks}\label{sect general frameworks}
\subsection{Root generating system}
A Kac-Moody matrix (or generalized Cartan matrix) is a square matrix $C=(c_{i,j})_{i,j\in I}$ with integers coefficients, indexed by  a finite set $I$ and satisfying: 
\begin{enumerate}
\item $\forall i\in I,\ c_{i,i}=2$

\item $\forall (i,j)\in I^2|i \neq j,\ c_{i,j}\leq 0$

\item $\forall (i,j)\in I^2,\ c_{i,j}=0 \Leftrightarrow c_{j,i}=0$.
\end{enumerate}

A root generating system is a $5$-tuple $\mathcal{S}=(C,X,Y,(\alpha_i)_{i\in I},(\alpha_i^\vee)_{i\in I})$ made of a Kac-Moody matrix $C$ indexed by $I$, of two dual free $\Z$-modules $X$ (of characters) and $Y$ (of cocharacters) of finite rank $\mathrm{rk}(X)$, a family $(\alpha_i)_{i\in I}$ (of simple roots) in $X$ and a family $(\alpha_i^\vee)_{i\in I}$ (of simple coroots) in $Y$. They have to satisfy the following compatibility condition: $c_{i,j}=\alpha_j(\alpha_i^\vee)$ for all $i,j\in I$. We also suppose that the family $(\alpha_i)_{i\in I}$ is free in $X$ and that the family $(\alpha_i^\vee)_{i\in I}$ is free in $Y$.

 We now fix a Kac-Moody matrix $C$ and a root generating system with matrix $C$.

Let $V=Y\otimes \R$. Every element of $X$ induces a linear form on $V$. We will consider $X$ as a subset of the dual $V^*$ of $V$: the $\alpha_i$, $i\in I$ are viewed as linear form on $V$. For $i\in I$, we define an involution $r_i$ of $V$ by $r_i(v)=v-\alpha_i(v)\alpha_i^\vee$ for all $v\in V$. Its space of fixed points is $\ker \alpha_i$. The subgroup of $\mathrm{GL}(V)$ generated by the $\alpha_i$ for $i\in I$ is denoted by $W^v$ and is called the Weyl group of $\mathcal S$.

For $x\in V$ we let $\underline{\alpha}(x)=(\alpha_i(x))_{i\in I}\in \R^I$.

Let  $Q^\vee=\bigoplus_{i\in I}\Z\alpha_i^\vee$ and $P^\vee=\{v\in V |\underline{\alpha}(v)\in \Z^I\}$. We call  $Q^\vee$ the \textit{coroot-lattice} and $P^\vee$ the \textit{coweight-lattice} (but if $\bigcap_{i\in I}\ker \alpha_i\neq \{0\}$, this is not a lattice). Let $Q_+^\vee=\bigoplus_{i\in I}\N \alpha_i^\vee$, $Q^\vee_-=-Q^\vee_+$ and $Q^\vee_{\mathbb{R}}=\bigoplus_{i\in I}\R\alpha_i^\vee$. This enables us to define a preorder $\leq_{Q^\vee}$ on $V$ by the following way: for all $x,y\in V$, one writes $x\leq_{Q^\vee}y$ if $y-x\in Q^\vee_+$.

One defines an action of the group $W^v$ on $V^*$ by the following way: if $x\in V$, $w\in W^v$ and $\alpha\in V^*$ then $(w.\alpha)(x)=\alpha(w^{-1}.x)$. Let $\Phi=\{w.\alpha_i|(w,i)\in W^v\times I\}$, $\Phi$ is the set of \textit{real roots}. Then $\Phi\subset Q$, where $Q=\bigoplus_{i\in I}\Z\alpha_i$. Let $W^a=Q^\vee\rtimes W^v\subset\mathrm{GA}(V\mathrm{})$ the \textit{affine Weyl group} of $\mathcal{S}$, where $\mathrm{GA}(V\mathrm{})$ is the group of affine isomorphisms of $V$.

Let $V_{in}=\bigcap_{i\in I}\ker \alpha_i$. Then one has $Y+V_{in}\subset P^\vee$.

\begin{rque}\label{rque changement de syst de racines}
Let $C=(c_{i,j})_{i,j\in I}$ a Kac-Moody matrix. Then there exists a generating root system $(C,X,Y,(\alpha_i)_{i\in I},(\alpha_i^\vee)_{i\in I})$ such that $Y+V_{in}=P^\vee$. For this, suppose $I=\llbracket 1,k\rrbracket$ for some $k\in \N$. Let $n=2k$ and let $(\alpha_i)_{i\in \llbracket 1,n\rrbracket }$ be the canonical basis of the dual of $\Z^n$. Let $(e_i)_{i\in \llbracket 1,n\rrbracket }$ be the canonical basis of $\Z^n$. For $i\in \llbracket 1,k\rrbracket$, let $\alpha_i^\vee=\sum_{j=1}^k c_{i,j}e_j+e_{i
+k}$. Let $X=\bigoplus_{i\in \llbracket 1,n\rrbracket }\Z\alpha_i=(\Z^n)^*$ and $Y=\Z^n$. Then $(C,X,Y,(\alpha_i)_{i\in I},(\alpha_i^\vee)_{i\in I})$ is a generating root system which is more or less $\mathcal{D}_{un}^C$ of 7.1.2 of \cite{remy2002groupes}.

Let $u\in P^\vee$. Then there exists $y\in Y$ such that $(\alpha_i(y))_{i\in I}=(\alpha_i(u))_{i\in I}$ and thus $u\in Y+V_{in}$. Therefore, $P^\vee=Y+V_{in}$.

\end{rque}

\subsection{Description of $Y^{++}$}\label{sect Y++}

In this subsection we show that $Y^{++}=Y\cap \overline{C^v_f}$ is a finitely generated monoid, which will be useful to prove Theorem~\ref{thm égalité des ensembles bis}.

Let $l\in \Ne$. Let us define a binary relation $\prec$ on $\N^l$. Let $x,y\in \N^l$,  and $x=(x_1,\ldots,x_l)$, $\ y=(y_1,\ldots,y_l)$, then one says $x\prec y$ if $x\neq y$ 
and for all $i\in \llbracket 1,l\rrbracket$, $x_i\leq y_i$.
\begin{lemma}\label{lemme ensembles d'incomparables}
Let $l\in \Ne$ and $F$ be a subset of $\N^l$ satisfying property (INC($l$)): 

for all $x,y\in F$, $x$ and $y$ are not comparable for $\prec$. 

 Then $F$ is finite.
\end{lemma}

\begin{proof} This  is clear for $l=1$ because a set $F$ satisfying INC($1$) is a singleton or $\emptyset$. 

Suppose that $l>1$ and that we have proved that all set satisfying INC($l-1$) is finite.

Let $F$ be a set  satisfying INC($l$) and suppose $F$ infinite. Let $(\lambda_n)_{n\in \N}$ be an injective sequence of $F$. One writes $(\lambda_n)=(\lambda_n^1,\ldots,\lambda_n^l)$. Let $M=\max \lambda_{0}$ and, for any $n\in \N$, $m_n=\min(\lambda_n)$. Then for all $n\in \N$, $m_n\leq M$ (if $m_n>M$, we would have $\lambda_{0}\prec \lambda_n$). Maybe extracting a sequence of $\lambda$, one can suppose that $(m_n)=(\lambda_n^i)$ for some $i\in \llbracket 1, l\rrbracket$ and that $(m_n)$ is constant and equal to $m_0\in \llbracket 0,M\rrbracket$. For $x\in \N^l$, we define $\tilde{x}=(x_j)_{j\in \llbracket 1,l\rrbracket\backslash \{i\}}\in \N^{l-1}$. 
 
 Set $\tilde{F}=\{\tilde{\lambda_n}|n\in\N\}$. The set $\tilde{F}$ satisfies INC($l-1$). By the induction hypothesis, $\tilde{F}$ is finite and thus $F$ is finite, which is absurd. Hence $F$ is finite and the lemma is proved.  
 \end{proof}

\begin{lemma}\label{lemme description de Y^{++}}
There exists a finite set $E\subset Y$ such that $Y^{++}=\sum_{e\in E}\N e$.
\end{lemma}

\begin{proof} The set $Y_{in}=Y\cap V_{in}$ is a lattice in the vectorial space it spans. Consequently, it is a finitely generated $\Z$-module and thus a finitely generated monoid. Let $E_1$ be a finite set generating $Y_{in}$ as a monoid.

 Let $Y_{\succ 0}=Y^{++}\backslash Y_{in}$. Let $\mathcal{P}=\{a\in Y_{\succ 0}|a\neq b+c\ \forall b,c\in Y_{\succ 0}\}$. Let $\alpha:Y^{++}\rightarrow \N^{I}$ such that 
 $\alpha(x)=(\alpha_i(x))_{i\in I}$ for all $x\in Y^{++}$. Let $a,b\in \mathcal{P}$.
  If $\alpha(a)\prec \alpha(b)$, then $b=b-a+a$, with $a,b-a\in Y_{\succ 0}$, which is absurd and by symmetry we deduce that $\alpha(a)$ and $\alpha(b)$ are not comparable for $\prec$. 
Therefore, by lemme~\ref{lemme ensembles d'incomparables}, $\alpha(\mathcal{P})$ is finite. Let $E_2$ be a finite set of $Y_{\succ 0}$ such that $\alpha(\mathcal{P})=\{\alpha(x)|x\in E_2\}$. Then $Y^{++}=\sum_{e\in E_2}\N e+Y_{in}=\sum_{e\in E}\N e$, where $E=E_1\cup E_2$.  
\end{proof}

\subsection{Vectorial faces}

Define $C_f^v=\{v\in V|\  \alpha_i(v)>0,\ \forall i\in I\}$. We call it the \textit{fundamental chamber}. For $J\subset I$, one sets $F^v(J)=\{v\in V|\ \alpha_i(v)=0\ \forall i\in J,\alpha_i(v)>0\ \forall i\in J\backslash I\}$. Then the closure $\overline{C_f^v}$ of $C_f^v$ is the union of the $F^v(J)$ for $J\subset I$. The \textit{positive} (resp. \textit{negative}) \textit{vectorial faces} are the sets $w.F^v(J)$ (resp. $-w.F^v(J)$) for $w\in W^v$  and $J\subset I$. A \textit{vectorial face} is either a positive vectorial face or a negative vectorial face. We call \textit{positive} (resp. \textit{negative}) \textit{chamber}  every cone  of the shape $w.C_f^v$ for some $w\in W^v$ (resp. $-w.C_f^v$).  For all $x\in C_f^v$ and for all $w\in W^v$, $w.x=x$ implies that $w=1$. In particular the action of $W^v$ on the positive chambers is simply transitive. The \textit{Tits cone} $\mathcal T$ is defined by $\mathcal{T}=\bigcup_{w\in W^v} w.\overline{C^v_f}$. We also consider the negative cone $-\mathcal{T}$.
We define a $W^v$ invariant preorder $\leq$ on $V$ by: $\forall (x,y)\in V\mathrm{}^2$, $x\leq y\ \Leftrightarrow\ y-x\in \mathcal{T}$.

\subsection{Filters}

\begin{definition}
A filter in a set $E$ is a nonempty set $F$ of nonempty subsets of $E$ such that, for all subsets $S$, $S'$ of $E$,  if $S$, $S'\in F$ then $S\cap S'\in F$ and, if $S'\subset S$, with $S'\in F$ then $S\in F$.
\end{definition}

If $F$ is a filter in a set $E$, and $E'$ is a subset of $E$, one says that $F$ contains $E'$ if every element of $F$ contains $E'$. If $E'$ is nonempty, the set $F_{E'}$ of subsets of $E$ containing $E'$ is a filter. By abuse of language, we will sometimes say that $E'$ is a filter by identifying $F_{E'}$ and $E'$. If $F$ is a filter in $E$, its closure $\overline F$ (resp. its convex envelope) is the filter of subsets of $E$ containing the closure (resp. the convex envelope) of some element of $F$. A filter $F$ is said to be contained in an other filter $F'$: $F\subset F'$ (resp. in a subset $Z$ in $E$: $F\subset Z$) if and only if any set in $F'$ (resp. if $Z$) is in $F$.

If $x\in V\mathrm{}$ and $\Omega$ is a subset of $V$ containing $x$ in its closure, then the \textit{germ} of $\Omega$ in $x$ is the filter $germ_x(\Omega)$ of subsets of $V$ containing a neighbourhood in $\Omega$ of $x$.

A \textit{sector} in $V$ is a set of the shape $\mathfrak{s}=x+C^v$ with $C^v=\pm w.C_f^v$ for some $x\in \mathbb{A}$ and $w\in W^v$. The point $x$ is its \textit{base point} and $C^v$ is its \textit{direction}. The intersection of two sectors of the same direction contains a sector of the same direction.

A \textit{sector-germ} of a sector $\mathfrak{s}=x+C^v$ is the filter $\mathfrak{S}$ of subsets of $V$ containing a $V$-translate of $\mathfrak{s}$. It only depends on the direction $C^v$. We denote by $+\infty$ (resp. $-\infty$) the sector-germ of $C_f^v$ (resp. of $-C_f^v$).

A ray $\delta$ with base point $x$ and containing $y\neq x$ (or the interval $]x,y]=[x,y]\backslash\{x\}$ or $[x,y]$) is called \textit{preordered} if $x\leq y$ or $y\leq x$ and \textit{generic} if $y-x\in \pm\mathring \T$, the interior of $\pm \T$.

In the next subsection, we define the notions of faces, enclosures and chimneys defined in \cite{rousseau2011masures} 1.7 and 1.10 and in  \cite{gaussent2014spherical} 1.4. For a first reading, one can just know the following facts about these objects and skip this subsection:

\begin{enumerate}

\item To any filter $F$ of $V$ is associated its enclosure $\mathrm{cl}_{\mathbb{A}}(F)$ which is a filter in $\mathbb{A}$ containing the convex envelope of the closure of $F$.

\item A face or a chimney is a filter in $V$.

\item A sector is a chimney which is solid and splayed.

\item If a chimney is a sector, its germ as a chimney coincides with its germ as a sector.

\item Every $x\in V\mathrm{}$ is in some face of $V$.

\item The group $W^a$ permutes the sectors, the enclosures, the faces and the chimneys of $V$.\label{fait sur les faces}

\end{enumerate}

\subsection{Definitions of enclosures, faces, chimneys and related notions}

 Let $\Delta=\Phi\cup\Delta_{im}^+\cup\Delta_{im}^-\subset Q$ be the set of all roots (recall that $Q=\bigoplus_{i\in I}\Z\alpha_i$) defined in  \cite{kac1994infinite}. The group $W^v$ stabilizes $\Delta$. For $\alpha\in \Delta$, and $k\in \Z\cup{+\infty}$, let $D(\alpha,k)=\{v\in V| \alpha(v)+k\geq 0\}$ (and $D(\alpha,+\infty)=V\mathrm{}$ for all $\alpha\in \Delta$) and $D^\circ(\alpha,k)=\{v\in V| \alpha(v)+k > 0\}$ (for $\alpha\in \Delta$ and $k\in \Z\cup\{+\infty\}$).

Given a filter $F$ of subsets of $V$, its \textit{enclosure} $\mathrm{cl}_V\mathrm{}(F)$ is the filter made of the subsets of $V$ containing an element of $F$ of the shape $\bigcap_{\alpha\in \Delta}D(\alpha,k_\alpha)$ where $k_\alpha\in \Z\cup\{+\infty\}$ for all $\alpha\in \Delta$.

A \textit{face} $F$ in $V$ is a filter associated to a point $x\in V\mathrm{}$ and a vectorial face $F^v\subset V$. More precisely, a subset $S$ of $V$ is an element of the face $F=F(x,F^v)$ if and only if, it contains an intersection of half-spaces $D(\alpha,k_\alpha)$ or open half-spaces $D^\circ(\alpha,k_\alpha)$, with $k_\alpha\in \Z$ for all $\alpha\in \Delta$, that contains $\Omega\cap (x+F^v)$, where $\Omega$ is an open neighbourhood of $x$ in $V$.

There is an order on the faces: if $F\subset \overline{F'}$ we say that "$F$ is a face of $F'$" or "$F'$ dominates $F$". The dimension of a face $F$ is the smallest dimension of an affine space generated by some $S\in F$. Such an affine space is unique and is called its support. A face is said to be \textit{spherical} if the direction of its support meets the open Tits cone $\mathring \T$ or its opposite $-\mathring \T$; then its pointwise stabilizer $W_F$ in $W^v$ is finite.

We have $W^a\subset P^\vee \rtimes W^v$. As $\alpha(P^\vee)\subset \Z$ for all $\alpha$ in $\bigoplus_{i\in I}\Z \alpha_i$, if $\tau$ is a translation of $V$ of a vector $p\in P^\vee$, then for all $\alpha\in Q$, $\tau$ permutes the sets of the shape $D(\alpha,k)$ where $k$ runs over $\Z$. As $W^v$ stabilizes $\Delta$, any element of $W^v$ permutes the sets of the shape $D(\alpha,k)$ where $\alpha$ runs over $\Delta$. Therefore, $W^a$ permutes the sets $D(\alpha,k)$, where $(\alpha,k)$ runs over $\Delta\times \Z$ and thus $W^a$ permutes the enclosures, faces, chimneys, ... of $V$.

A \textit{chamber} (or alcove) is a maximal face, or equivalently, a face such that all its elements contains a nonempty open subset of $V$.

A \textit{panel} is a spherical face maximal among faces that are not chambers or, equivalently, a spherical face of dimension $n-1$.

A \textit{chimney} in $V$ is associated to a face $F=F(x,F_0^v)$ and to a vectorial face $F^v$; it is the filter $\mathfrak{r}(F,F^v)=\mathrm{cl}_\mathbb{A}(F+F^v)$. The face $F$ is the basis of the chimney and the vectorial face $F^v$ its direction. A chimney is \textit{splayed} if $F^v$ is spherical, and is \textit{solid} if its support (as a filter, i.e., the smallest affine subspace of $V$  containing $\mathfrak{r}$) has a finite pointwise stabilizer in $W^v$. A splayed chimney is therefore solid. 

A \textit{shortening} of a chimney $\mathfrak{r}(F,F^v)$, with $F=F(x,F_0^v)$ is a chimney of the shape $\mathfrak{r} (F(x+\xi,F_0^v),F^v)$ for some $\xi\in \overline{F^v}$. The \textit{germ} of a chimney $\mathfrak{r}$ is the filter of subsets of $V$ containing a shortening of $\mathfrak{r}$ (this definition of shortening is slightly different from the one of \cite{rousseau2011masures} 1.12 but follows \cite{rousseau2017almost} 3.6) and we obtain the same germs with these two definitions).

\subsection{Masure}
We now denote by $\A$ the affine space $V$ equipped with its faces, chimneys, ...

An apartment of type $\A$ is a set $A$ with a nonempty set $\mathrm{Isom}(\A,A)$ of bijections (called isomorphisms) such that if $f_0\in \mathrm{Isom}(\A,A)$ then $f\in \mathrm{Isom}(\A,A)$ if and only if, there exists $w\in W^a$ satisfying $f=f_0\circ w$. An isomorphism between two apartments $\phi:A\rightarrow A'$ is a bijection such that ($f\in \mathrm{Isom}(\mathbb{A},A)$ if, and only if, $\phi \circ f\in \mathrm{Isom}(\A,A')$). We extend all the notions that are preserved by $W^a$ to each apartment. By the fact~\ref{fait sur les faces} of the above subsection, sectors, enclosures, faces and chimneys are well defined in any apartment of type $\A$.

\begin{definition}
An ordered affine masure of type $\A$ (also called a masure of type $\A$) is a set $\mathcal{I}$ endowed with a covering $\mathcal{A}$ of subsets called apartments such that: 

(MA1) Any $A\in \mathcal{A}$ admits a structure of an apartment of type $\A$.

(MA2) If $F$ is a point, a germ of a preordered interval, a generic ray or a solid chimney in an apartment $A$ and if $A'$ is another apartment containing $F$, then $A\cap A'$ contains the enclosure $\mathrm{cl}_A(F)$ of $F$ and there exists an isomorphism from $A$ onto $A'$ fixing $\mathrm{cl}_A(F)$.

(MA3) If $\mathfrak{R}$ is the germ of a splayed chimney and if $F$ is a face or a germ of a solid chimney, then there exists an apartment that contains $\mathfrak{R}$ and $F$.

(MA4) If two apartments $A$, $A'$ contain $\mathfrak{R}$ and $F$ as in (MA3), then there exists an isomorphism from $A$ to $A'$ fixing $\mathrm{cl}_A(\mathfrak{R}\cup F)$.

(MAO) If $x$, $y$ are two points contained in two apartments $A$ and $A'$, and if $x\leq_{A} y$ then the two segments $[x,y]_A$ and $[x,y]_{A'}$ are equal.
\end{definition}

In this definition, we say that an apartment contains a germ of a filter if it contains at least one element of this germ. We say that an application fixes a germ if it fixes at least one element of this germ.

\begin{rque}(consequence 2.2 3) of \cite{rousseau2011masures})\label{rque axioms MA3 et MA4 modifiés}
By (MA2), the axioms (MA3) and (MA4) also apply in a masure when $F$ is a point, a germ of a preodered segment and when $\mathfrak{R}$ or $F$ is a germ of a generic ray or a germ of a spherical sector face.
\end{rque}

\medskip
Until the end of this article, $\I$ will be a masure We suppose that $\I$ is thick of \textit{finite thickness}: the number of chambers (=alcoves) containing a given panel has to be finite, greater or equal to $3$. This assumption will be crucial to use some theorems of \cite{gaussent2014spherical} but we will not use it directly. 

We assume that $\I$ has a strongly transitive group of automorphisms $G$, which means that all isomorphisms involved in the above axioms are induced by elements of $G$. We choose in $\I$ a fundamental apartment, that we identify with $\A$. As $G$ is strongly transitive, the apartments of $\I$ are the sets $g.\A$ for $g\in G$. The stabilizer $N$ of $\A$ induces a group $\nu(N)$ of affine automorphisms of $\A$ and we suppose that $\nu(N)=W^v\ltimes Y$.

An example of such a masure $\I$ is the masure associated to a split Kac-Moody group over a ultrametric field constructed in \cite{gaussent2008kac} and in \cite{rousseau2017almost}. We will precise it in Subsection~\ref{SubsecMasure associée à un groupe de Kac-Moody}

\paragraph{Preorder and vectorial distance}
As the preorder $\leq$ on $\A$ (induced by the Tits cone) is invariant under the action of $W^v$, we can equip each apartment $A$ with a preorder $\leq_A$.  Let $A$ be an apartment of $\I$ and $x,y\in A$ such that $x\leq_A y$. Then by Proposition 5.4 of \cite{rousseau2011masures}, if $B$ is an apartment containing $x$ and $y$, $x\leq_B y$. As a consequence, one can define a relation $\leq$ on $\I$ as follow: let $x,y\in \I$, one says that $x\leq y$ if there exists an apartment $A$ of $\I$ containing $x$ and $y$ and such that $x\leq_A y$.  By Théorème 5.9 of \cite{rousseau2011masures}, this defines a $G$-invariant preorder on $\I$.

For $x\in \mathcal{T}$, we denote by $x^{++}$ the unique element in $\overline{C^v_f}$ conjugated by $W^v$ to $x$. 

 Let $\I\times_{\leq}\I=\{(x,y)\in \I^2|x\leq y\}$ be the set of increasing pairs in $\I$. Let $(x,y)$ be such a pair and $g\in G$ such that $g.x,g.y\in \A$. Then $g.x\leq g.y$ and we define the \textit{vectorial distance} $d^v(x,y)\in\overline{C_f^v}$ by $d^v(x,y)=(g.y-g.x)^{++}$. It does not depend on the choices we made. This "distance" is $G$-invariant.

For $x\in \I$ and $\lambda\in \overline{C_f^v}$, one defines $S^v(x,\lambda)=\{y\in \I|x\leq y\mathrm{\ and\ }d^v(x,y)=\lambda\}$.

\begin{rque}\label{rque caractérisation distance vectorielle}
a) If $a\in Y$ and $\lambda\in \overline{C_f^v}$, then $S^v(a,\lambda)=\{x\in \I|\exists g\in G|g.a=a\mathrm{\ and\ }g.x=a+\lambda\}$.

b) Let $x,y\in \I$ and suppose that for some $g\in G$, $g.y-g.x\in \overline{C_f^v}$. Then $x\leq y$ and $d^v(x,y)=g.y-g.x$.

\end{rque}

\subsection{Retractions and Hecke paths}

Let  $\mathfrak{R}$ be the germ of a splayed chimney of an apartment $A$. Let $x\in \I$. By (MA3), for all $x\in \I$, there exists an apartment $A_x$ of $\I$ containing $x$ and $\mathfrak{R}$. By (MA4), there exists an isomorphism of apartments $\phi:A_x\rightarrow A$ fixing $\mathfrak{R}$. By \cite{rousseau2011masures} 2.6, $\phi(x)$ does not depend on the choices we made and thus we can set $\rho_{A,\mathfrak{R}}(x)=\phi(x)$.

The application $\rho_{A,\mathfrak{R}}$ is a retraction from $\I$ onto $A$. It only depends on $\mathfrak{R}$ and $A$ and we call it the \textit{retraction onto $A$ centred at $\mathfrak{R}$}. 

We denote by $\rho_{+\infty}$ (resp. $\rho_{-\infty}$ ) the retraction onto $\A$ centred at $+\infty$ (resp. $-\infty$).

We now define Hecke paths. They are more or less the images by $\rho_{-\infty}$ of preordered segments $[x,y]$ in $\I$. The definition is a bit technical but it expresses the fact that the image of such a path "goes nearer to $+\infty$" when it crosses a wall. A consequence of that is Remark~\ref{rque chemins de Hecke} and we will not use directly this definition in the following.

 We consider piecewise linear continuous paths $\pi:[0,1]\rightarrow \A$ such that the values  of $\pi'$ belong to some orbit $W^v.\lambda$ for some $\lambda\in \overline{C_f^v}$. Such a path is called a $\lambda$-\textit{path}. It is increasing with respect to the preorder relation $\leq$ on $\A$. For any $t\neq 0$ (resp. $t\neq 1$), we let $\pi'_-(t)$ (resp. $\pi'_+(t))$ denote the derivative of $\pi$ at $t$ from the left (resp. from the right).

\begin{definition}
A Hecke path of shape $\lambda$ with respect to $-C_f^v$ is a $\lambda$-path such that
 $\pi'_+(t)\leq_{W^v_{\pi (t)}} \pi'_-(t)$ for all $t\in [0,1]\backslash \{0,1\}$, which 
 means that there exists a $W_{\pi(t)}^v$-chain from $\pi'_-(t)$ to $\pi'_{+}(t)$, i.e., a 
 finite sequence $(\xi_0=\pi'_-(t),\xi_1,\ldots, \xi_s=\pi'_+(t))$ of vectors in $V$ and
  $(\beta_1,\ldots,\beta_s)\in \Phi^s$ such that, for all $i\in \llbracket 1,s\rrbracket$,
\begin{enumerate}
\item $r_{\beta_i}(\xi_{i-1})=\xi_i.$

\item $\beta_i(\xi_{i-1})<0.$

\item $r_{\beta_i}\in W^v_{\pi(t)}$; i.e., $\beta_i(\pi(t))\in \Z$: $\pi(t)$ is in a wall of direction $\ker(\beta_i)$.

\item Each $\beta_i$ is positive with respect to $-C_f^v$; i.e., $\beta_i(C_f^v)>0$.
\end{enumerate}
\end{definition}

\begin{rque}\label{rque chemins de Hecke}
Let $\pi:[0,1]\rightarrow \A$ be a Hecke path of shape $\lambda\in \overline{C_f^v}$ with respect to $-C^v_f$. Then if $t\in [0,1]$ such that $\pi$ is differentiable in $t$ and $\pi'(t)\in \overline{C_f^v}$, then for all $s\geq t$, $\pi$ is differentiable in $s$ and $\pi'(s)=\lambda$.
\end{rque}

\subsection{Masure associated to a split Kac-Moody group}\label{SubsecMasure associée à un groupe de Kac-Moody}
 Let $\mathbf{G}$ be a split Kac-Moody group over a ultrametric field $\mathcal F$ with ring of integer $\mathcal O$ and  $G=\mathbf{G}(\mathcal{F})$.  We use notation of the introduction.  We associate a masure to $G$ as in  \cite{gaussent2008kac} and \cite{rousseau2017almost}.  Let us explain the dictionnary, given in the introduction,  between objects in $G$ and objects in the $\I$.

 The group $G$ acts strongly transitively on $\I$.  We have $Q=\Lambda$ and $Q^\vee=\Lambda^\vee$. The group $K$ is the fixer of $0$ in $G$ and $HU$ (resp. $HU^-$) is the fixer of $+\infty$ (resp. $-\infty$) in $G$, where $H=\mathbf{T}(\mathcal{O})$.  The action of $T$ on $\A$ is as follows: if $t=\pi^{\lambda^\vee}\in T$ for some $\lambda^\vee\in \Lambda^\vee$, $t$ acts on $\A$ by the translation of vector $-\lambda^\vee$ and thus $\lambda^\vee=\pi^{-\lambda^\vee}.0$ for all $\lambda^\vee\in \Lambda^\vee$.

Let $\I_0=G.0$ be
 the set of \textit{vertices of type} $0$. The  map $G\rightarrow \I_0$ sending $g\in G$ to $g.0$ induces a bijection $\phi: G/K \rightarrow \I_0$. If $g\in G$, then $\phi^{-1}(g.0)=gK$. If $x\in \A$, then  $\rho_{+\infty}^{-1}(\{x\})=U.x$ and $\rho_{-\infty}^{-1}(\{x\})=U^-.x$.  By Remark~\ref{rque caractérisation distance vectorielle}, if $\lambda^\vee\in \Lambda^\vee$, then $S^v(0,\lambda^\vee)=K.\lambda^\vee$.

Therefore, for all $\lambda^\vee,\mu^\vee\in \Lambda^{\vee}$, $\phi^{-1}\big(S^v(0,\lambda^\vee)\cap \rho_{+\infty}^{-1}(\{\mu^\vee\})\big)=(K\pi^{-\lambda^\vee}K\cap U\pi^{-\mu^\vee}K)/K$ and $\phi^{-1}\big(\rho_{-\infty}^{-1}(\{\lambda^\vee\})\cap \rho_{+\infty}^{-1}(\{\mu^\vee\})\big)=(U^-\pi^{-\lambda^\vee}K\cap U\pi^{-\mu^\vee}K)/K$.

We then use the bijection $\psi:K\backslash G \rightarrow G/K$ defined by $\psi(Kg)=g^{-1}K$ to obtain the sets considered in the introduction.

\section{Segments and Rays in $\I$}\label{sect preliminaries}

In this section we begin by defining for all $\nu\in C_f^v$ two applications $y_\nu:\I\rightarrow \A$ and $T_\nu:\I\rightarrow \R_+$, where for all $x\in \I$, $T_\nu(x)$ and $y_\nu(x)$ can be considered as the distance between $x$ and $\A$ along $\R_+\nu$ and the projection of $x$ on $\A$ along $\R_+\nu$. 

We also show that the only antecedent of some paths for $\rho_{-\infty}$ are themselves (this is Lemma~\ref{lemme image réciproque de segments}).

\medskip

Let $Q^\vee_{\R_+}=\bigoplus_{i\in I}\R_+\alpha_i^\vee$ and $Q^\vee_{\R_-}=-Q^\vee_{\R_+}$. Recall that we want to prove the following theorem: 

\begin{theorem}\label{thm inclusion}
Let $\mu\in \A$. Then if $\mu\notin Q_{\R_-}^\vee$, $\rho_{+\infty}^{-1}(\{\lambda+\mu\})\cap \rho_{-\infty}^{-1}(\{\lambda\})$ is empty for all $\lambda\in \A$. If $\mu\in Q^\vee_{\R_-
}$, then for $\lambda\in \A$ sufficiently dominant, $\rho_{+\infty}^{-1}(\{\lambda+\mu\})\cap \rho_{-\infty}^{-1}(\{\lambda\})\subset S^v(0,\lambda)\cap \rho_{+\infty}^{-1}(\{\lambda+\mu\}) $. 
\end{theorem}

This theorem is a bit more general than its statement in the introduction because $\lambda,\mu\in \A$ are arbitrary and not only in $Q^\vee$. It will be proved in Section~\ref{sect bounding of T}.

\medskip
Let us recall briefly the notion of parallelism in $\I$. This is done more completely in \cite{rousseau2011masures} Section 3. Let $\delta$ and $\delta'$ be two generic rays in $\I$. Then there exists a splayed chimney $R$ containing $\delta$ and a solid chimney $F$ containing $\delta'$. By (MA3) there exists an apartment $A$ containing the germ $\mathfrak{R}$ of $R$ and $F$. Therefore $A$ contains translates of $\delta$ and $\delta'$ and we say that $\delta$ and $\delta'$ are \textit{parallel}, if these translates are parallel in $A$. Parallelism is an equivalence relation and its equivalence classes are called \textit{directions}.

\begin{lemma}\label{lemme demi-droite de base donnée}
Let $x\in \I$ and $\delta$ be a generic ray. Then there exists a unique ray $x+\delta$ in $\I$ with base point $x$ and direction $\delta$. In any apartment $A$ containing $x$ and a ray $\delta'$ parallel to $\delta$, this ray is the translate in $A$ of $\delta'$ having $x$ as a base point.
\end{lemma}

This lemma is analogous to Proposition 4.7 1) of \cite{rousseau2011masures}. The difficult part of this lemma is the uniqueness of such a ray because second part of the lemma yields a way to construct a ray having direction $\delta$ and $x$ as a base point. This uniqueness can be shown exactly in the same manner as the proof of Proposition 4.7.1) by replacing "spherical sector face" by "generic ray". This is possible by NB.a) of Proposition 2.7 and by 2.2 3) (or by Remark~\ref{rque axioms MA3 et MA4 modifiés} of this paper) of \cite{rousseau2011masures}.

\paragraph{Definition of $y_\nu$ and $T_\nu$ (resp. $y^-_\nu$ and $T^-_\nu$)}
Let $x\in \mathcal{I}$. Let $\nu \in C_f^v$ and $\delta=\R_+\nu$, which is a generic ray. According to axiom (MA3) applied to a face containing $x$ and the splayed chimney $C^v_f$, there exists an apartment $A$ containing $x$ and $+\infty$. Then $A$ contains $x+\delta$. The set $x+\delta\cap \A$ is nonempty. Let $z\in x+\delta\cap \A$. Then $A\cap \A$ contains $z$, $+\infty$ and by (MA4), $A\cap\A$ contains $\mathrm{cl}(z,+\infty)$. As $\mathrm{cl}(z,+\infty)\supset z+\overline{C_f^v}$,  $A\cap\A \supset z+\delta$ and thus $x+\delta\cap \A=y+\delta$ or $x+\delta\cap \A=y+\mathring{\delta}$ for some $y\in x+\delta$, where $\mathring \delta=\R^*_+\nu$.

 Suppose $x+\delta\cap \A=y+\mathring{\delta}$. Let $z\in y+\mathring{\delta}$. Then by (MA2) applied to $germ_y([y,z]\backslash\{y\})$, $\A\cap A\supset \mathrm{cl}(germ_y([y,z]\backslash\{y\}))\ni y$ because $\mathrm{cl}(germ_y([y,z]\backslash\{y\}))$ contains the closure of $germ_y([y,z]\backslash\{y\})$. This is absurd and thus $\A\cap x+\delta=y+\delta$, with $y\in \A$. One sets $y_\nu(x)=y\in\A$ (actually, $y_\nu$ only depends on $\delta$).

One has $\rho_{+\infty}(x+\delta)=\rho_{+\infty}(x)+\delta$ and $y\in \rho_{+\infty}(x)+\delta$. We define $T_\nu(x)$ as the unique element $T$ of $\mathbb{R}_+$ such that $y=\rho_{+\infty}(x)+T\nu$.

Let $\delta^-=-\R_+\nu$ and $x\in \I$. Similarly, one defines $y_\nu^-$ as the first point of $x+\delta^-$ meeting $\A$ and $T^-_\nu(x)$ as the element $T$ of $\R_+$ such that $\rho_{-\infty}(x)=y+T\nu$.

The proof of Theorem~\ref{thm inclusion} will rely on the fact that $T_\nu^-$ is bounded by some function of $\rho_{+\infty}-\rho_{-\infty}$, which is Corollary~\ref{corollaire majoration de T}. This bounding will be obtained by studying Hecke paths.

\begin{rque}
In the following, the choice of $\nu$ will not be very important. We will often need to choose $\nu\in Y\cap C^v_f$.
\end{rque}

\begin{example}
Let us describe the action of $y_\nu$ on a simple example. Let $\mathbf{G}$ be as in Subsection~\ref{SubsecMasure associée à un groupe de Kac-Moody} (we keep  notation as in the introduction). For $\alpha\in \Phi$, let $U_\alpha$ be the root subgoup associated to $\alpha$ and $x_\alpha:(\mathcal F,+)\rightarrow U_\alpha$ be  an isomorphism of algebraic group. Let $\omega:\mathcal F\rightarrow \Z\cup\{+\infty\}$ be a surjective valuation inducing the structure of ultrametric field of $\mathcal F$.  Define $\phi_\alpha: U_\alpha\rightarrow \R$ by $\phi_\alpha(u)=\omega(x_\alpha^{-1}(u))$ for all $u\in U_\alpha$.  For $k\in \R\cap\{+\infty\}$, one sets $U_{\alpha,k}=\phi_{\alpha}^{-1}([k,+\infty])$ and $D(\alpha,k)=\{x\in \A|\alpha(x)+k\geq 0\}$. Let $H=\mathbf T (\mathcal O)$. Then by 4.2 5) and 4.2 7) of \cite{gaussent2008kac}, $H$ is the fixer of $\A$ in $\I$ and for all $\in \R$,  $HU_{\alpha,k}$ is the fixer of  $D(\alpha,k)$ in $\I$.

Let $k\in \Z$, $u\in HU_{\alpha,k}\backslash HU_{\alpha,k+1}$, and $A=u.\A$. One has $A\cap \A=\{x\in \A|u.x=x\}=\{u\in \A|u.x\in \A\}$. In order to simplify, suppose that $\mathbf G$ is a reductive group (and thus $\I$ is a usual Bruhat-Tits building).  Let us show that  $D(\alpha,k)=A\cap \A$ (actually this remains true if $G$ is not a reductive group but the proof is a bit more difficult and we do not want to develop it here).

By 2.5.7 of  \cite{bruhat1972groupes}, $A\cap \A$ is enclosed (which means that the enclosure of $A\cap \A$ is $A\cap\A$) and thus $A\cap \A=D(\alpha,l)$ for some $l\in [k,+\infty)\cap \Z$. As $u\in HU_{\alpha,l}$, we deduce that $k=l$. 

One writes $\A=M\oplus \R\nu$, where $M=\{x\in \A| \alpha(x)=-k\}$. Then if  $x=m+\lambda\nu$, with $m\in M$ and $\lambda\in \R$, $y_\nu(u.x)=m+\mu(\lambda)\nu$, with $\mu(\lambda)=\min(\lambda,0)$.
\end{example}

\begin{lemma}\label{lemme distance vectorielle}
Let $x\in \mathcal{I}$ and $\nu\in C_f^v$. Let $y=y_\nu(x)$ and $T=T_\nu(x)$.

a)  Then $x\leq y$ and $d^v(x,y)=T\nu$.

b) One has $\rho_{+\infty}(x)\in Y$ if and only if $\rho_{-\infty}(x)\in Y$ if and only if $x\in \mathcal I_0$. In this case, $\rho_{+\infty}(x)\leq_{Q^\vee} \rho_{-\infty}(x)$.

\end{lemma}

\begin{proof} Let $A$ be an apartment containing $x$ and $+\infty$ and $g\in G$ fixing $+\infty$ such that $A=g^{-1}.\A$. Then $x+\delta$ is the translate of a shortening $\delta'\subset A$ of $\delta$ (which means $\delta'=z+\delta$, with $z\in \delta$). As for all $z'\in z+\delta$, $z\leq z'$, one has $x\leq y$. As $d^v(x,y)=d^v(g.x,g.y)$ and $g_{|A}=\rho_{+\infty}$ one gets a).

For $x\in\mathcal{I}$, there exists $g_-,g_+\in G$ such that $\rho_{-\infty}(x)=g_-.x$ and $\rho_{+\infty}(x)=g_+.x$, which shows the claimed equivalence because $Y=G.0\cap \mathbb{A}$.

Suppose $x\in \mathcal{I}_0$. One chooses $\nu\in Y\cap C_f^v$. Let $S=\lfloor T\rfloor +1$, where $\lfloor.\rfloor$ is the floor function, and $z=\rho_{+\infty}(x)+S\nu\in x+\delta$. Then $d^v(x,z)=d^v(g_+.x,g_+.z)=d^v(\rho_{+\infty}(x),z)=S\nu\in Y\cap C^v_f$.

According to paragraph 2.3 of \cite{gaussent2014spherical}, the image $\pi$ of $[x,z]$ by $\rho_{-\infty}$ is a Hecke path of shape $z-\rho_{+\infty}(x)=S\nu$ with respect to $-C^v_f$ (unless the contrary is specified, "Hecke path" will mean with respect to $-C^v_f$). By applying Lemma 2.4b) of \cite{gaussent2014spherical} to $\pi$, one gets that $z-\rho_{-\infty}(x)\leq _{Q^\vee} d^v(x,z)=z-\rho_{+\infty}(x)$ and thus $\rho_{+\infty}(x)\leq_{Q^\vee} \rho_{-\infty}(x)$ and one has b).  
\end{proof}

\begin{lemma}\label{lemme image réciproque de segments}
Let $\tau:[0,1]\rightarrow \mathcal{I}$ be a segment of $\mathcal{I}$ such that $\tau(1)\in \mathbb{A}$ and $\rho_{-\infty}\circ\tau$ is a segment of $\mathbb{A}$ satisfying  $(\rho_{-\infty}\circ\tau)'=\nu\in \overline{C^v_f}$. Then $\tau([0,1])\subset \mathbb{A}$ and thus $\rho_{-\infty}\circ \tau=\tau$.
\end{lemma}

\begin{proof} Suppose $\tau([0,1])\not\subset \A$. Let $u=\sup\{t\in [0,1]|\tau(t)\notin \A\}$. Then by the same reasoning as in the proof that "$x+\delta\cap\A=y+\delta$" in the paragraph "Definition of $y_\nu$ and $T_\nu$", $x=\tau(u)\in\A$. One has $\tau(0)\leq x\leq \tau(1)$ (by the same reasoning as in the proof of Lemma \ref{lemme distance vectorielle} a)). 

By Remark~\ref{rque axioms MA3 et MA4 modifiés}, there exists an apartment $A=g^{-1}.\A$ with $g\in G$ containing $\mathfrak{R}=-\infty$ and $germ_x([x,\tau(0)])$. By axiom (MA4) (and Remark~\ref{rque axioms MA3 et MA4 modifiés}) applied to $\mathfrak{R}=-\infty$ and to $x$, we can suppose that $g$ fixes $\mathrm{cl}(x,-\infty)\supset x-C_f^v$.  
Let $x'\in [x,\tau(0)]\backslash\{x\}$ such that $[x,x']\subset A$ and $x'\notin \A$. Then $g.x'=\rho_{-\infty}(x')\in x-\R_+\nu\subset x-C_f^v$. Therefore, $x'=g.x'\in \A$, which is absurd. Hence $\tau([0,1])\subset \A$.  
\end{proof}

\section{Bounding of $T_\nu$ and Proof of Theorem~\ref{thm inclusion} }\label{sect bounding of T}

One defines $h: Q^\vee_{\mathbb{R}} \rightarrow \mathbb{R}$ by $h(x)=  \sum_{i\in I} x_i$, for all $x=\sum_{i\in I} x_i\alpha_i^\vee\in Q^\vee_{\R}$.

\begin{lemma}\label{lemme fin des chemins de Hecke longs}

Let $T\in \mathbb{R}_+$, $\mu\in \A$,  $a\in \mathbb{A}$, $\nu\in Y^{++}=Y\cap\overline{C^v_f}$ and suppose there exists a Hecke path $\pi$ from $a$ to $a+T\nu-\mu$ of shape $T\nu$. Then 

a) $\mu\in Q^\vee_{\R_+}$. Consequently $h(\mu)$ is well defined.

b) if $T>h(\mu)$, there exists $t$ such that $\pi$ is differentiable on $(t,1]$ and $\pi'_{|(t,1]}=T\nu$. Furthermore, let $t^*$ be the smallest $t \in [0,1]$ having this property, then $t^*\leq \frac{h(\mu)}{T}$.
\end{lemma}

\begin{proof} The main idea of b) is to use the fact that during the time when $\pi'(t)\neq T\nu$, $\pi'(t)=T\nu-T\lambda(t)$ with $\lambda(t)\in Q^\vee_+\backslash \{0\}$. Hence for $T$ large, $\pi$ decreases quickly for the $Q^\vee$ order, but it cannot decrease too much because $\mu$ is fixed.

Let $t_0=0$, $t_1, \ldots,t_n=1$ be a subdivision of $[0,1]$ such that for all $i\in \llbracket 0,n-1\rrbracket $, $\pi_{|(t_i,t_{i+1})}$ is differentiable and let $w_i\in W^v$ be such that $\pi'_{|(t_i,t_{i+1})}=w_i.T\nu$. If $w_i.\nu=\nu$, one chooses $w_i=1$.

For $i\in \llbracket 0,n-1\rrbracket$, according to Lemma 2.4 a) of \cite{gaussent2014spherical}, $w_i.\nu =\nu -\lambda_i$, with $\lambda_i \in Q^\vee_+$ and if $w_i\neq 1$, $\lambda_i\neq 0$. One has \[\pi(1)-\pi(0)=T\nu-\sum_{i=0, w_i\neq 1}^{n-1}(t_{i+1}-t_i)T\lambda_i=T\nu -\mu\]

and one deduces a).

Suppose now $T>h(\mu)$. Let us show that there exists $i\in \llbracket 0,n-1\rrbracket$ such that $w_i=1$. Let $i\in \llbracket 0,n-1\rrbracket$. 

For all $i$ such that $w_i\neq 1$, one has $h(\lambda_i)\geq 1$. Hence $T\sum_{i=0, w_i\neq 1}^{n-1}(t_{i+1}-t_i)\leq h(\mu)$, and $\sum_{i=0, w_i\neq 1}^{n-1}(t_{i+1}-t_i)<1=\sum_{i=0}^{n-1}(t_{i+1}-t_i)$. Thus there exists $i\in \llbracket 0,n-1\rrbracket$ such that $w_i=1$.

By Remark~\ref{rque chemins de Hecke}, if $w_i=1$ for some $i$, then $w_j=1$ for all $j\geq i$. This shows the existence of $t^*$. We also have $t^*\leq \sum_{i=0, w_i\neq 1}^{n-1}(t_{i+1}-t_i)$ and hence the claimed inequality follows.  
\end{proof}

\medskip

From now on and until the end of this subsection, $\nu$ will be a fixed element of $C_f^v\cap Y$. We define $T_\nu:\I\rightarrow \R_+$ and $y_\nu:\I\rightarrow \A$ as in the paragraph "Definition of $y_\nu$ and $T_\nu$".

\begin{corollary}\label{corollaire majoration de T}
Let $x\in \I$ and $\mu=\rho_{-\infty}(x)-\rho_{+\infty}(x)$. Then $\mu\in Q^\vee_{\R_+}$ and $T_\nu(x)\leq h(\mu)$.
\end{corollary}

\begin{proof} Let $y=y_\nu(x)$ and $T=T_\nu(x)$. Let $\pi$ be the image by $\rho_{-\infty}$ of $[x,y]$. This is a Hecke path from $\rho_{-\infty}(x)$ to $y=\rho_{+\infty}(x)+T\nu$, of shape $T\nu$. The minimality of $T$ and Lemma \ref{lemme image réciproque de segments} imply that $\pi'(t)\neq \nu$ for all $t\in [0,1]$ where $\pi$ is differentiable. By applying Lemma \ref{lemme fin des chemins de Hecke longs}b), we deduce that $T\leq h(\mu)$.  Lemma~\ref{lemme fin des chemins de Hecke longs}a) applied to $a=\rho_{-\infty}(x)$ and $\mu=\rho_{-\infty}(x)-\rho_{+\infty}(x)$ completes the proof. 
\end{proof}

\begin{rque}\label{rque majoration de T^-}
With the same reasoning but by considering Hecke with respect to $C^v_f$ we obtain an analogous bounding for $T_\nu^-$: for all $x\in \I$, $T_\nu^-(x)\leq h(\rho_{-\infty}(x)-\rho_{+\infty}(x))$.
\end{rque}

\begin{corollary}\label{cor points ayant meme image par les rétractions}
 Let $x\in\mathcal{I}$ such that $\rho_{+\infty}(x)=\rho_{-\infty}(x)$. Then $x\in \mathbb{A}$. Therefore, $\forall z\in \mathbb{A}$, $\rho_{+\infty}^{-1}(\{z\})\cap\rho_{-\infty}^{-1}(\{z\})=\{z\}$.
\end{corollary}

\begin{proof}  Let $x\in\mathcal{I}$ such that $\rho_{+\infty}(x)=\rho_{-\infty}(x)$. Then by Corollary~\ref{corollaire majoration de T}, $T_\nu(x)=0$ and thus $x\in\A$.   
\end{proof}

\begin{lemma}\label{lemme_rétraction et distance vectorielle}
Let $x\in \mathcal{I}$ such that $y_\nu^-(x)\in C^v_f$. Then $0\leq x$ and $\rho_{-\infty}(x)=d^v(0,x)$.
\end{lemma}

\begin{proof} Let $y^-=y^-_\nu(x)$. Let $A$ be an apartment containing $x$ and $-\infty$. By (MA4) there exists $g\in G$ such that $A=g^{-1}.\A$ and $g$ fixes $\mathrm{cl}(y,-\infty)\supset y-\overline{C_f^v}\ni 0$. Then $g.x-g.y^-=\rho_{-\infty}(x)-y^-=T^-\nu\in C_f^v$ and $g.y^--g.0=y^-$. Thus $g.x-g.0=\rho_{-\infty}(x)\in C_f^v$ and we can conclude by Remark~\ref{rque caractérisation distance vectorielle}.  
\end{proof}

We can now prove Theorem~\ref{thm inclusion}: 

Let $\mu\in \A$. Then if $\mu\notin Q_{\R_-}^\vee$, $\rho_{+\infty}^{-1}(\{\lambda+\mu\})\cap \rho_{-\infty}^{-1}(\{\lambda\})$ is empty for all $\lambda\in \A$. If $\mu\in Q^\vee_{\R_-
}$, then for $\lambda\in \A$ sufficiently dominant, $\rho_{+\infty}^{-1}(\{\lambda+\mu\})\cap \rho_{-\infty}^{-1}(\{\lambda\})\subset S^v(0,\lambda)\cap \rho_{+\infty}^{-1}(\{\lambda+\mu\}) $. 

\begin{proof} The condition on $\mu$ comes from Corollary~\ref{corollaire majoration de T}. 

Suppose $\mu\in Q^\vee_{\R_-}$.
Let $\lambda\in \A$. Let $y^-=y^-_\nu$ and $T^-=T_\nu^-$. By Corollary~\ref{corollaire majoration de T} and remark~\ref{rque majoration de T^-}, if $x\in \rho_{+\infty}^{-1}(\{\lambda+\mu\})\cap \rho_{-\infty}^{-1}(\{\lambda\})$, then $y^-(x)\in [\lambda-T^-(x)\nu, \lambda]\subset \lambda-[0,h(-\mu)]\nu$. 
For all $i\in I$, $\alpha_i([0,h(-\mu)]\nu)$ is bounded. Consequently for $\lambda$ sufficiently dominant, $\alpha_i(\lambda-[0,h(-\mu)]\nu)\subset \mathbb{R}^*_+$ for all $i\in I$.  For such a $\lambda$,  $y^-\big(\rho_{+\infty}^{-1}(\{\lambda+\mu\})\cap \rho_{-\infty}^{-1}(\{\lambda\})\big)\subset C^v_f$. We conclude the proof with Lemma \ref{lemme_rétraction et distance vectorielle}.
  
\end{proof}

\section{Study of "Translations" of $\I$ and Proof of Theorem~\ref{thm invariance des cardinaux}}\label{sect translations}

Let $\A_{in}=\bigcap_{i\in I} \ker \alpha_i$.

In this subsection we introduce some kind of "translations" of $\I$ of an inessential vector and show that they have very nice properties. It will be useful to "generalize theorems from $Y$ to $Y+\A_{in}$" by getting rid of the inessential part. First example of this technique will be Corollary~\ref{cor finitude des boules} which generalizes a theorem of \cite{gaussent2014spherical}. Then we study elements of $G$ inducing a translation on $\A$. We show that they commute with $\rho_{+\infty}$ and $\rho_{-\infty}$. We then can see that for fixed $\mu\in Q^\vee$, the $\rho_{+\infty}^{-1}(\{\lambda+\mu\})\cap\rho_{-\infty}^{-1}(\{\lambda\})$, for $\lambda\in Y+\A_{in}$, are some translates of  $\rho_{+\infty}^{-1}(\{\mu\})\cap\rho_{-\infty}^{-1}(\{0\})$, which enables to show Theorem~\ref{thm invariance des cardinaux} by using Theorem~\ref{thm inclusion} and Corollary~\ref{cor finitude des boules}.

\begin{lemma}\label{lemme partie inessentielle}
 Let $\nu\in\A_{in}$ and $a\in \I$. Then $|S^v(a,\nu)|=1$. Moreover, if $a\in \A$, $S^v(a,\nu)=\{a+\nu\}$. 

\end{lemma} 

\begin{proof} Let $h\in G$ such that $h.a\in \A$. Then $S^v(a,\nu)=h^{-1}.S^v(h.a,\nu)$. Therefore one can assume $a\in \A$. 

Let $x\in S^v(a,\nu)$. Let $g\in G$ such that $x,a\in g^{-1}.\A$ and $g.x-g.a=\nu$. Let $\tau:[0,1]\rightarrow \I$ defined by $\tau(t)=g^{-1}.(g.a+(1-t)\nu)$. Then $\pi=\rho_{-\infty}\circ \tau$ is a Hecke path of shape $-\nu$ and in particular,  it is a $-\nu$-path. For all $t$ where $\pi$ is differentiable, there exists $w(t)\in W^v$ such that $\pi'(t)=-w(t).\nu=-\nu$. As $W^v$ acts trivially on $\A_{in}$, $\pi'(t)=-\nu$ for all such $t$ and thus $\pi$ is differentiable on $[0,1]$ and $\pi'=-\nu$. As $\tau(1)=a\in \A$, one can apply Lemma~\ref{lemme image réciproque de segments} and we get that $\tau([0,1])\subset \A$. Therefore, $x\in \A$ and there exists $w\in W^v$ such that $x-a=w.\nu=\nu$.  
\end{proof}

\medskip

This lemma enables us to define  a kind of translation of an inessential vector. Let $\nu\in \A_{in}$. Let $\tau_\nu:\I\rightarrow \I$ which associates to $x\in \I$ the unique element of $S^v(x,\nu)$. Then we have the following lemma: 

\begin{lemma}\label{lemme propriétés des translations}
\begin{enumerate}
Let $\nu\in \A_{in}$ and $\tau=\tau_\nu$. Then:

\item For all $x\in \A$, $\tau(x)=x+\nu$.\label{item tau restreint à A}

\item For all $g\in G$, $g\circ\tau=\tau\circ g$. In particular, for all $x\in \I$, if $x$ is in an apartment $A$,
 then $\tau(x)\in A$, and if $x=g.a$ with $g\in G$ and $a\in\A$, then $\tau(x)=g.(x+\nu)$.\label{item commutation de tau}

\item The map $\tau$ is a bijection, its inverse being $\tau_{-\nu}$.\label{item bijectivité des translations}

\item The map $\tau$ commutes with $\rho_{+\infty}$ and $\rho_{-\infty}.$\label{item commutation translation retraction}

\item Let $x\in \I$ and $\lambda\in \overline{C_f^v}$, then $\tau(S^v(x,\lambda))=S^v(\tau(x),\lambda)=S^v(x,\lambda+\nu).$\label{item translation d'une boule}

\end{enumerate}
\end{lemma}

\begin{proof} Part~\ref{item tau restreint à A} is a part of Lemma~\ref{lemme partie inessentielle}.

Choose $x\in \I$ and $g\in G$. Then $d^v(x,\tau(x))=\nu$ and thus $d^v(g.x,g.\tau(x))=\nu$. Consequently $g.\tau(x)=\tau(g.x)$ by definition of $\tau$. Let $A$ be an apartment containing $x$, $A=h.\A$, with $h\in G$. Then $\tau(x)=\tau(h.a)$ with $a\in \A$, hence $\tau(x)=h.(x+\nu)\in A=h.\A$ and we have~\ref{item commutation de tau}.

Let $x\in\I$, $x=g.a$ with $a\in \A$. By part~\ref{item commutation de tau} applied to $\tau$ and $\tau_{-\nu}$, one has $\tau_{-\nu}(\tau(g.a))=g.\tau_{-\nu}(\tau(a))=g.a$ and thus $\tau_{-\nu}\circ\tau=\Id$. This is enough to show ~\ref{item bijectivité des translations}.

Let $x\in \I$ and $g\in G$ fixing $+\infty$ such that $g.x=\rho_{+\infty}(x)$. 
Then $\tau(x)\in g^{-1}.\A$, thus $g.\tau(x)=\rho_{+\infty}(\tau(x))$ and by
 part~\ref{item commutation de tau}, $g.\tau(x)=\tau(g.x)$. Hence, $\tau$ and
  $\rho_{+\infty}$ commute and by the same reasoning, this is also true for $\tau$ and $\rho_{-\infty}$.

Let $x\in\I$ and $\lambda\in \overline{C_f^v}$. Let $u\in S^v(x,\lambda)$. There exists $g'\in G$ such that $x,u\in g'^{-1}.\A$ and $g'.u-g'.x=w.\lambda$, for some $w\in W^v$. Let $n\in N$ inducing $w$ on $\A$ and $g=n^{-1}g'$. Then $x,u\in g^{-1}.\A$ and $g.u-g.x=\lambda$. We have $g.\tau(u)-g.x=\tau(g.u)-g.x=\lambda+\nu$. Therefore, $\tau(S^v(x,\lambda))\subset S^v(x,\lambda+\nu)$. Applying this result with $\tau_{-\nu}$ yields $\tau_{-\nu}(S^v(x,\lambda+\nu))\subset S^v(x,\lambda)$ and thus $\tau(S^v(x,\lambda))=S^v(x,\lambda+\nu)$.

One has $g.\tau(u)-g.\tau(x)=(g.u+\nu)-(g.x+\nu)=\lambda$ and thus $\tau(S^v(x,\lambda))\subset S^v(\tau(x),\lambda)$. Again, by considering $\tau_{-\nu}$, we have that $\tau(S^v(x,\lambda))=S^v(\tau(x),\lambda)$.  
\end{proof}

\begin{corollary}\label{cor finitude des boules}
Let $\lambda\in \A$ and $\mu\in Y+\A_{in}$. Then for all $\lambda_{in}\in \A_{in}$, $\tau_{\lambda_{in}}\big(S^v(0,\lambda)\cap \rho_{+\infty}^{-1}(\{\mu\})\big)=S^v(0,\lambda+\lambda_{in})\cap \rho_{+\infty}^{-1}(\{\mu+\lambda_{in}\})$. In particular, for all $\lambda\in Y+\A_{in}$, $S^v(0,\lambda)\cap \rho_{+\infty}^{-1}(\{\mu+\lambda\})$ is finite and is empty if $\mu\notin Q^\vee_-$.

\end{corollary}

\begin{proof} The first assertion is a consequence of Lemma~\ref{lemme propriétés des translations} part~\ref{item bijectivité des translations}, \ref{item commutation translation retraction} and \ref{item translation d'une boule}.

Let $\lambda\in Y+\A_{in}$ and $\lambda_{in}\in \A_{in}$ such that $\tau(\lambda)\in Y$, 
with $\tau=\tau_{\lambda_{in}}$. Then 
$\tau\big(S^v(0,\lambda)\cap\rho_{+\infty}^{-1}(\{\lambda+\mu\})\big)=S^v(0,\tau(\lambda))\cap\rho_{+\infty}^{-1} (\{\tau(\lambda+\mu )\})$. Consequently, one can assume $\lambda\in Y$. 

Suppose $S^v(0,\lambda)\cap \rho_{+\infty}^{-1}(\{\mu+\lambda\})$ is nonempty. Let $x$ be in this set. Then there exists $g,h\in G$ such that $g.x=\lambda$ and $h.x=\mu+\lambda$. Thus $\lambda+\mu=h.g^{-1}.\lambda\in \I_0\cap \A=Y$ and therefore, $\mu\in Y$. We can now conclude because the finiteness and the condition on  $\mu$ are shown in \cite{gaussent2014spherical}, Section 5: if $\lambda, \mu\in Y$ then $S^v(0,\lambda)\cap \rho_{+\infty}^{-1}(\{\mu\})$ is finite and it is empty if $\mu\notin Q^\vee_-$ (the cardinals of these sets correspond to the $n_\lambda(\nu)$ of loc. cit.).
 
\end{proof}

We now show a lemma similar to Lemma~\ref{lemme propriétés des translations} part~\ref{item commutation translation retraction} for translations of $G$: 

\begin{lemma}\label{lemme commutation des translations et rétractions}
Let $n\in G$ inducing a translation on $\A$. Then $n\circ\rho_{+\infty}=\rho_{+\infty}\circ n$ and $n\circ \rho_{-\infty}=\rho_{-\infty}\circ n$.
\end{lemma}

\begin{proof} 
Let $x\in\mathcal{I}$ and $A$ be an apartment containing $x$ and $+\infty$. Then $n.A$ is an apartment containing $+\infty$. Let $\phi:A\rightarrow \mathbb{A}$ an isomorphism fixing $+\infty$. We have $n.x\in n.A$, and $n\circ\phi\circ n^{-1}:n.A\rightarrow \mathbb{A}$ fixes $+\infty$. Hence $\rho_{+\infty}(n.x)=n\circ\phi\circ n^{-1}(n.x)=n\circ\phi(x)=n\circ\rho_{+\infty}(x)$ and thus $n\circ \rho_{+\infty}=\rho_{+\infty}\circ n$. By the same reasoning applied to $\rho_{-\infty}$, we get the lemma.  
 \end{proof}

\begin{lemma}\label{lemme y- des translatés}
Let $n\in G$ inducing a translation on $\mathbb{A}$. Let $\lambda_{in}\in \A_{in}$. Set $\tau=\tau_{\lambda_{in}}\circ n$. Let $\nu\in C_f^v$ and $y^-=y_\nu^-$. Then $\tau\circ y^-=y^-\circ \tau$.
\end{lemma}

\begin{proof} If $x\in \mathbb{A}$, then $y^-(x)=x$, $y^-(\tau(x))=\tau(x)$ and there is nothing to prove. 

Suppose $x\notin \mathbb{A}$. 
Then $[x,y^-(x)]\backslash \{y^-(x)\}\subset (x-\R_+\nu)\backslash \mathbb{A}$,
 thus $\tau([x,y^-(x)]\backslash\{y^-(x)\})\subset (\tau(x)-\R_+\nu)\backslash \mathbb{A}$ and $\tau(y^-(x))\in \mathbb{A}$.  
 \end{proof}

\begin{theorem}\label{thm invariance des cardinaux}
Let $\mu\in \A$ and $\lambda\in Y+\A_{in}$. One writes $\lambda=\lambda_{in}+\Lambda$, with $\lambda_{in}\in\A_{in}$ and $\Lambda\in Y$. Let $n\in G$ inducing the translation of vector $\Lambda$ and $\tau=\tau_{\lambda_{in}}\circ n$. Then $\rho_{+\infty}^{-1}(\{\lambda+\mu\})\cap\rho_{-\infty}^{-1}(\{\lambda\})=n\big(\rho_{+\infty}^{-1}(\{\mu\})\cap\rho_{-\infty}^{-1}(\{0\})\big)$. Therefore these sets are finite and if $\mu\notin Q^\vee_-$ these sets are empty.

\end{theorem}

\begin{proof} First assertion is a consequence of Lemma~\ref{lemme commutation des translations et rétractions} and Lemma~\ref{lemme propriétés des translations} part~\ref{item commutation translation retraction}. Then Lemma~\ref{lemme distance vectorielle} b) shows that these sets are empty unless $\mu\in Q^\vee_-$. 

By Theorem~\ref{thm inclusion}, for $\lambda'\in Y$ sufficiently dominant, $\rho_{+\infty}^{-1}(\{\lambda'+\mu\})\cap\rho_{-\infty}^{-1}(\{\lambda'\})\subset S^v(0,\lambda')\cap \rho_{+\infty}^{-1}(\{\lambda'+\mu\})$, which is a finite set by \cite{gaussent2014spherical} Section 5 (or by Corollary~\ref{cor finitude des boules}).  

\end{proof}

\section{Proof of Theorem~\ref{thm égalité des ensembles bis}}\label{sect proof of final theorem}
Recall that we want to prove the following theorem:

 \begin{theorem}\label{thm égalité des ensembles bis}
 Let $\mu\in Q^\vee$. Then for $\lambda\in Y^{++}+\A_{in}$ sufficiently dominant, $S^v(0,\lambda)\cap \rho_{+\infty}^{-1}(\{\lambda+\mu\})=\rho_{-\infty}^{-1}(\{\lambda\})\cap\rho_{+\infty}^{-1}(\{\lambda+\mu\})$.
 \end{theorem}

The proof is postponed to the end of this seciton. The basic idea of this proof is that if $\mu\in Q^\vee$ there exists a finite set $F\subset Y^{++}$ such that for all $\lambda\in Y^{++}+\A_{in}$, $\rho_{+\infty}^{-1}(\{\lambda+\mu\})\cap S^v(0,\lambda) \subset \bigcup_{f\in F|\lambda-f\in\overline{C^v_f}} \rho_{+\infty}^{-1}(\{\lambda+\mu\})\cap S^v(\lambda-f,f)$ (this is Corollary~\ref{lemme majoration du cardinal des boules}, which generalizes Lemma~\ref{lemme distance finie à l'appartement} and uses Section~\ref{sect Y++}). Then we use Subsection~\ref{sect translations} to show  that $\rho_{+\infty}^{-1}(\{\lambda+\mu\})\cap S^v(\lambda-f,f)$ is the image of $\rho_{+\infty}^{-1}(\mu+f)\cap S^v(0,f)$ by a "translation" $\tau_{\lambda-f}$ of $G$ of vector $\lambda-f$ (which means that $\tau_{\lambda-f}$ induces the translation of vector $\lambda-f$ on $\mathbb{A}$). We fix $\nu\in C_f^v$ and set $y^-=y_\nu^-$. By Lemma~\ref{lemme y- des translatés},  \[y^-\big(\rho_{+\infty}^{-1}(\{\lambda+\mu\})\cap S^v(0,\lambda)\big)\subset\bigcup_{f\in F}\tau_{\lambda-f}\circ y^-\big(S^v(0,f)\cap \rho_{+\infty}^{-1}(\{\mu +f\})\big).\] According to Section 5 of \cite{gaussent2014spherical}, for all $f, \mu \in Y$, $S^v(0,f)\cap \rho_{+\infty}^{-1}(\{\mu +f\})$ is finite. Consequently, for $\lambda$ sufficiently dominant, $\bigcup_{f\in F}\tau_{\lambda-f}\big(y^-(S^v(0,f)\cap \rho_{+\infty}^{-1}(\{\mu +f\}\big)\subset C^v_f$  and one concludes with Lemma~\ref{lemme_rétraction et distance vectorielle}.

\begin{lemma}\label{lemme distance finie à l'appartement}
Let $\mu\in Q^\vee_-$ and $H=-h(\mu)+1\in \N$. Let $a\in Y$, $T\in [H,+\infty)$, $\nu\in Y^{++}$ and $x\in S^v(a,T\nu)\cap\rho_{+\infty}^{-1}(\{a+T\nu+\mu\})$. Let $g\in G$ such that $g.a=a$ and $g.x=T\nu+a$. Then $g$ fixes $[a,a+(T-H)\nu]$ and in particular $x\in S^v(a+(T-H)\nu,H\nu)$.
\end{lemma}

\begin{proof} Let $\tau:[0,1]\rightarrow \mathbb{A}$ defined by $\tau(t)=a+(1-t)T\nu$. The main idea is to apply Lemma \ref{lemme fin des chemins de Hecke longs} to $\rho_{+\infty}\circ (g^{-1}.\tau)$ but we cannot do it directly because $\rho_{+\infty}\circ\tau$ is not a Hecke path with respect to $-C^v_f$. Let $\mathbb{A}'$ be the vectorial space $\mathbb{A}$ equipped with a structure of apartment of type $-\mathbb{A}$: the fundamental chamber of $\mathbb{A}'$ is $C'^v_f=-C^v_f$ etc ... Let $\mathcal{I}'$ be the set $\mathcal{I}$, whose apartments are the $-A$ where $A$ runs over the apartment of $\I$. Then $\mathcal{I}'$ is a masure of standard apartment $\mathbb{A}'$. We have $a\leq x$ in $\mathcal{I}$ and so $x\leq'a$ in $\mathcal{I}'$. 
 
Then the image $\pi$ of $g^{-1}.\tau$ by $\rho_{+\infty}$ is a Hecke path of shape $-T\nu$ from $\rho_{+\infty}(x)=a+T\nu +\mu$ to $a$. By Lemma~\ref{lemme fin des chemins de Hecke longs} (for $\I'$), for $t>-h(\mu)/T$, $\pi'(t)=-T\nu$, and thus Lemma~\ref{lemme image réciproque de segments} (for $\I'$ and $\A'$) implies $\rho_{+\infty}(g^{-1}.\tau(t))=g^{-1}.\tau(t)$ for all $t>-h(\mu)/T$. Therefore, $g^{-1}.\tau_{|(-h(\mu)/T,1]}$ is a segment in $\A$ ending in $a$, with derivative $-T\nu$ and thus, for all $t\in (\frac{-h(\mu)}{T},1]$, $g^{-1}.\tau(t)=\tau(t)$.
 In particular, $g$ fixes $[a,\tau(\frac{H}{T})]=[a,a+(T-H)\nu]$, and $d^v(a+(T-H)\nu,x)=d^v(g^{-1}.(a+(T-H)\nu),g^{-1}.x)=d^v(a+(T-H)\nu,a+T\nu)=H\nu$.  
\end{proof}

Let $E$ be as in the Lemma~\ref{lemme description de Y^{++}}.

\begin{lemma}\label{lemme distance finie à l'appartement bis}
Let $\mu\in Q^\vee_-$, $H=-h(\mu)+1\in \N$ and $a\in Y$. Let $\lambda\in Y^{++}$. One writes $\lambda=\sum_{e\in E}\lambda_e e$ with $\lambda_e\in\N$ for all $e\in E$. Let $e\in E$. Then if $\lambda_e\geq H$, $S^v(a,\lambda)\cap\rho_{+\infty}^{-1}(\{a+\lambda+\mu\})\subset S^v(a+(\lambda_e -H)e,\lambda-(\lambda_e-H)e)$. 
\end{lemma} 

\begin{proof} Let $x\in S^v(a,\lambda)\cap\rho_{+\infty}^{-1}(\{a+\lambda+\mu\})$ and $g\in G$ fixing $a$ such that $g.x=a+\lambda$.  Let $z=g^{-1}(a+\lambda_e e)$.

Then one has $d^v(a,z)=\lambda_e e$ and $d^v(z,x)=\lambda-\lambda_e e$.

According to Lemma 2.4.b) of \cite{gaussent2014spherical} (adapted because one considers Hecke paths with respect to $C^v_f$), one has: 

$\rho_{+\infty}(z)-a\leq_{Q^\vee} d^v(a,z)=\lambda_e e$ and $\rho_{+\infty}(x)-\rho_{+\infty}(z)\leq_{Q^\vee} d^v(z,x)=\lambda-\lambda_e e$.

Therefore, \[a+\lambda+\mu=\rho_{+\infty}(x)\leq_{Q^\vee} \rho_{+\infty}(z)+\lambda-\lambda_e e\leq_{Q^\vee}a+\lambda.\]

Hence, $\rho_{+\infty}(z)=a+\lambda_e e+\mu'$, with $\mu\leq_{Q^\vee}\mu'\leq_{Q^\vee}0$. One has $-h(\mu')+1\leq H$. By Lemma~\ref{lemme distance finie à l'appartement}, $g$ fixes $[a,a+(\lambda_e-H)e]$ and thus $g$ fixes $a+(\lambda_e-H)e$. 

As $d^v(g^{-1}.(a+(\lambda_e-H)e),x)=\lambda-(\lambda_e-H)e$, $x\in S^v(a+(\lambda_e -H)e,\lambda-(\lambda_e-H)e)$.  
 \end{proof}

\begin{corollary}\label{lemme majoration du cardinal des boules}
 Let $\mu \in Q^\vee_-$. Let $H=-h(\mu)+1$. Let $\lambda\in Y^{++}$. We fix a writing  $\lambda=\sum_{e\in E}\lambda_e e$, with $\lambda_e\in \N$ for all $e\in E$. Let $J=\{e\in E | \lambda_e \geq H\}$. Then $S^v(0,\lambda)\cap \rho_{+\infty}^{-1}(\{\lambda+\mu\})\subset S^v(\lambda-H\sum_{e\in J} e-\sum_{e\notin J}\lambda_e e, H\sum_{e\in J}e+\sum_{e\notin J}\lambda_ e)$. 
\end{corollary}
 
\begin{proof} This is a generalization by induction of Lemma~\ref{lemme distance finie à l'appartement bis}.
\end{proof}

 We now prove Theorem~\ref{thm égalité des ensembles bis}: 

 Let $\mu\in Q^\vee$. Then for $\lambda\in Y^{++}+\A_{in}$ sufficiently dominant, $S^v(0,\lambda)\cap \rho_{+\infty}^{-1}(\{\lambda+\mu\})=\rho_{-\infty}^{-1}(\{\lambda\})\cap\rho_{+\infty}^{-1}(\{\lambda+\mu\})$.

\begin{proof} Theorem \ref{thm inclusion} yields one inclusion. It remains to show that $\rho_{+\infty}^{-1}(\{\lambda+\mu\})\cap S^v(0,\lambda)\subset\rho_{+\infty}^{-1}(\{\lambda+\mu\})\cap \rho_{-\infty}^{-1}(\{\lambda\})$ 
for $\lambda$ sufficiently dominant.

Let $H=-h(\mu)+1$ and $F=\{\sum_{e\in E}\nu_e e|(\nu_e)\in \llbracket 0,H\rrbracket^E\}$. This set is finite. Let $\lambda\in Y^{++}+\A_{in}$, $\lambda=\lambda_{in}+\Lambda$, with $\lambda_{in}\in \A_{in}$ and $\Lambda\in Y^{++}$.

Let $x\in S^v(0,\lambda)\cap \rho_{+\infty}^{-1}(\{\lambda+\mu\})$. Then by Corrollary~\ref{lemme majoration du cardinal des boules}, there exists $f\in F$ such that $\lambda-f\in \overline{C^v_f}$ and $x\in S^v(\lambda-f,f)$ (one can take $f=H\sum_{e\in J}e+\sum_{e\notin J}\lambda_ e$ where $J$ is as in Corollary~\ref{lemme majoration du cardinal des boules}). Let $n$ be an element of $G$ inducing the translation of vector $\Lambda-f=\lambda-\lambda_{in}-f$ on $\A$ and $\tau_{\lambda,f}=\tau_{\lambda_{in}}\circ n$. Then $x\in\tau_{\lambda,f}(B_f)$ where $B_f=S^v(0,f)\cap \rho_{+\infty}^{-1}(\{\mu +f\})$.

Let $B=\bigcup_{f\in F}B_f$. Then one has proven that \[S^v(0,\lambda)\cap \rho_{+\infty}^{-1}(\{\lambda+\mu\})\subset \bigcup_{f\in F}\tau_{\lambda,f}(B).\]

By Section 5 of \cite{gaussent2014spherical} (or Corollary~\ref{cor finitude des boules}), $B_f$ is finite for all $f\in F$ and thus $B=\bigcup_{f\in F}B_f$ is finite.  Let $\nu\in C_f^v$ and $y^-=y_\nu^-$. Then $y^-(B)$ is finite and for $\lambda$ sufficiently dominant, $\bigcup_{f\in F}\tau_{\lambda,f}\circ y^-(B) \subset{C^v_f}$. Moreover, according to Lemma~\ref{lemme y- des translatés}, $\bigcup_{f\in F}\tau_{\lambda,f}\circ y^-(B)=\bigcup_{f\in F}y^-\circ \tau_{\lambda,f}(B)$. Hence \[y^-\big(S^v(0,\lambda)\cap \rho_{+\infty}^{-1}(\{\lambda+\mu\})\big)\subset C^v_f\] for $\lambda$ sufficiently dominant. Eventually one concludes with Lemma~\ref{lemme_rétraction et distance vectorielle}.  
\end{proof}

\bibliography{bibliographie}
\bibliographystyle{plain}
 \end{document}